\documentclass{LMCS}

\def\dOi{10(2:9)2014}
\lmcsheading%
{\dOi}
{1--10}
{}
{}
{Jan.~\phantom04, 2013}
{Jun.~16, 2014}
{}

\ACMCCS{[{\bf Theory of computation}]: Constructive Mathematics}

\usepackage{enumerate}
\usepackage{hyperref}
\usepackage{amssymb}
\def\dbigcup{\mathop{\displaystyle \bigcup }}
\def\func#1{\mathop{\rm #1}}

\theoremstyle{plain}

\begin{document}

\title{Locating $\mathfrak{A}x$, where $\mathfrak{A}$ is a subspace of $\mathcal{B}(H)$}

\author{Douglas S. Bridges}	
\address{University of Canterbury, Christchurch, New Zealand}	
\email{douglas.bridges@canterbury.ac.nz}  





\begin{abstract}
  Given a linear space of operators on a Hilbert space, any vector in
  the latter determines a subspace of its images under all operators.
  We discuss, within a Bishop-style constructive framework, conditions
  under which the projection of the original Hilbert space onto the
  closure of the image space exists.  We derive a general result that
  leads directly to both the open mapping theorem and our main theorem
  on the existence of the projection.
\end{abstract}

\maketitle


\section{Introduction}

Let $H$ be a real or complex Hilbert space, $\mathcal{B}(H)$ the space of bounded operators on $H$, and $\mathfrak{A}$ a linear subspace of $\mathcal{B}(H)$. For each $x\in H$ write
\[\mathfrak{A}x\equiv \left\{ Ax:A\in \mathfrak{A}\right\} , \]
and, \emph{if it exists}, denote the projection of $H$ onto the closure $\overline{\mathfrak{A}x}$ of $\mathfrak{A}x$ by $\left[ \mathfrak{A}x\right]$. Projections of this type play a very big part in the classical theory of operator algebras, in which context $\mathfrak{A}$ is normally a subalgebra of $\mathcal{B}(H)$; see, for example, \cite{Dixmier,KR,Sakai,Topping}. However, in the constructive\footnote{Our \emph{constructive setting} is that of Bishop \cite{Bishop,BB,BV}, in which the mathematics is developed with intuitionistic, not classical, logic, in a suitable set- or type-theoretic framework \cite{Aczel,ML} and with dependent choice permitted.} setting---the one of this paper---we cannot even guarantee that $\left[ \mathfrak{A}x\right] $ exists. Our aim is to give sufficient conditions on $\mathfrak{A}$ and $x$ under which $\left[\mathfrak{A}x\right] $ exists, or, equivalently, the set $\mathfrak{A}x$ is located, in the sense that
\[\rho \left( v,\mathfrak{A}x\right) \equiv \inf \left\{ \left\Vert v-Ax\right\Vert :A\in \mathfrak{A}\right\} \]
exists for each $v\in H$.

We require some background on operator topologies. Specifically, in addition to the standard uniform topology on $\mathcal{B}(H)$, we need
\begin{itemize}
\item[$\vartriangleright $] the \emph{\textbf{strong operator topology:}} the weakest topology on $\mathcal{B}(H)$ with respect to which the mapping $T\rightsquigarrow Tx$ is continuous for all $x\in H$;
\item[$\vartriangleright $] the \textbf{\emph{weak operator topology:}} the weakest topology on $\mathcal{B}(H)$ with respect to which the mapping $T\rightsquigarrow \left\langle Tx,y\right\rangle $ is continuous for all $x,y\in H$.
\end{itemize}
These topologies are induced, respectively, by the seminorms of the form $T\rightsquigarrow \left\Vert Tx\right\Vert $ with $x\in H$, and $T\rightsquigarrow \left\vert \left\langle Tx,y\right\rangle \right\vert$ with $x,y\in H$. The unit ball\footnote{Note that it is not constructively provable that every element $T$ of $\mathcal{B}(H)$ is normed, in the sense that the usual operator norm of $T$ exists. Nevertheless, when we write `$\left\Vert T\right\Vert \leqslant 1$', we are using a shorthand for `$\left\Vert Tx\right\Vert \leqslant \left\Vert x\right\Vert$ for each $x\in H$'. Likewise, `$\left\Vert T\right\Vert <1$' means that there exists $c<1$ such that $\left\Vert Tx\right\Vert \leqslant c\left\Vert x\right\Vert $ for each $x\in H$; and `$\left\Vert T\right\Vert>1$' means that there exists $x\in H$ such that $\left\Vert Tx\right\Vert
>\left\Vert x\right\Vert $.}
\[\mathcal{B}_{1}(H)\equiv \left\{ T\in \mathcal{B}(H):\left\Vert T\right\Vert\leqslant 1\right\}\]
of $\mathcal{B}(H)$ is classically weak-operator compact, but constructively the most we can say is that it is weak-operator totally bounded (see \cite{BVwo}). The evidence so far suggests that in order to make progress when dealing constructively with a subspace or subalgebra $\mathfrak{A}$ of $ \mathcal{B}(H)$, it makes sense to add the weak-operator total boundedness of
\[\mathfrak{A}_{1}\equiv \mathfrak{A}\cap \mathcal{B}_{1}(H)\]
to whatever other hypothesis we are making; in particular, it is known that $\mathfrak{A}_{1}$ is located in the strong operator topology---and hence $\mathfrak{A}_{1}x$ is located for each $x\in H$---if and only if it is weak-operator totally bounded \cite{BHV,Spitters}.

Recall that the \emph{\textbf{metric complement}} of a subset $S$ of a metric space $X$ is the set $-S$ of those elements of $X$ that are bounded away from $X$. When $Y$ is a subspace of $X$, $y\in Y$, and $S\subset Y$, we define
\[\rho _{Y}\left( y,-S\right) \equiv \inf \left\{ \rho \left( y,z\right) :z\in Y\cap -S\right\}\]
if that infimum exists.

We now state our main result.

\begin{thm}\label{2202a}
Let $\mathfrak{A}$ be a uniformly closed subspace of $\mathcal{B}(H)$ such that $\mathfrak{A}_{1}$ is weak-operator totally bounded, and let $x$ be a point of $H$ such that $\mathfrak{A}x$ is closed and\thinspace $\rho _{\mathfrak{A}x}\left( 0,-\mathfrak{A}_{1}x\right) $ exists. Then the projection $\left[ \mathfrak{A}x\right] $ exists.
\end{thm}
Before proving this theorem, we discuss, in Section \ref{sec:2}, some general results about the locatedness of sets like $\mathfrak{A}x$, and we derive, in Section 3, a generalisation of the open mapping theorem that leads to the proof of Theorem \ref{2202a}. Finally, we show, by means of a Brouwerian example, that the existence of $\rho _{\mathfrak{A}x}\left( 0,-\mathfrak{A}_{1}x\right) $ cannot be dropped from the hypotheses of our main theorem.

\section{Some general locatedness results for $\mathfrak{A}x$}\label{sec:2}

We now prove an elementary, but helpful, result on locatedness in a Hilbert space.

\begin{prop}\label{2202b}
Let $\left( S_{n}\right) _{n\geqslant 1}$ be a sequence of located, convex subsets of a Hilbert space $H$ such that $S_{1}\subset S_{2}\subset \cdots $ , let $S_{\infty }=\dbigcup\limits_{n\geqslant 1}S_{n}$, and let $x\in H$. For each $n$, let $x_{n}\in S_{n}$ satisfy $\left\Vert x-x_{n}\right\Vert <\rho \left( x,S_{n}\right) +2^{-n}$. Then
\begin{equation}
\rho \left( x,S_{\infty }\right) =\inf_{n\geqslant 1}\rho(x,S_{n})=\lim_{n\rightarrow \infty }\rho \left( x,S_{n}\right),\label{6}
\end{equation}
in the sense that if any of these three numbers exists, then all three do and they are equal. Moreover, $\rho \left( x,S_{\infty }\right) $ exists if and only if $\left( x_{n}\right) _{n\geqslant 1}$ converges to a limit $x_{\infty }\in H$; in that case, $\rho \left( x,S_{\infty }\right)=\left\Vert x-x_{\infty }\right\Vert $, and $\left\Vert x-y\right\Vert>\left\Vert x-x_{\infty }\right\Vert $ for all $y\in S_{\infty }$ with $y\neq x_{\infty }$.
\end{prop}

\begin{proof}
Suppose that $\rho \left( x,S_{\infty }\right) $ exists. Then $\rho \left(x,S_{\infty }\right) \leqslant \rho \left( x,S_{n}\right) $ for each $n$. On the other hand, given $\varepsilon >0$ we can find $z\in S_{\infty }$ such that $\left\Vert x-z\right\Vert <\rho \left( x,S_{\infty }\right)+\varepsilon $. Pick $N$ such that $z\in S_{N}$. Then for all $n\geqslant N$,
\[\rho \left( x,S_{\infty }\right) \leqslant \rho \left( x,S_{n}\right)\leqslant \rho \left( x,S_{N}\right) \leqslant \left\Vert x-z\right\Vert<\rho \left( x,S_{\infty }\right) +\varepsilon.\]
The desired conclusion (\ref{6}) now follows.

Next, observe that (by the parallelogram law in $H$) if $m\geqslant n$, then
\begin{eqnarray*}
\left\Vert x_{m}-x_{n}\right\Vert ^{2} &\leqslant &\left\Vert \left(x-x_{m}\right) -\left( x-x_{n}\right) \right\Vert ^{2} \\
&=&2\left\Vert x-x_{m}\right\Vert ^{2}+2\left\Vert x-x_{n}\right\Vert^{2}-4\left\Vert x-\frac{1}{2}\left( x_{m}+x_{n}\right) \right\Vert ^{2} \\
&\leqslant &2\left( \rho \left( x,S_{m}\right) +2^{-m}\right) ^{2}+2\left(\rho \left( x,S_{n}\right) +2^{-n}\right) ^{2}-4\rho \left( x,S_{m}\right)^{2},
\end{eqnarray*}
since $\frac{1}{2}\left( x_{m}+x_{n}\right) \in S_{m}$. Thus
\begin{eqnarray}
\left\Vert x_{m}-x_{n}\right\Vert ^{2} &\leqslant &2\left( \left( \rho\left( x,S_{m}\right) +2^{-m}\right) ^{2}-\rho \left( x,S_{m}\right)^{2}\right)  \nonumber \\
&&+2\left( \left( \rho \left( x,S_{n}\right) +2^{-n}\right) ^{2}-\rho \left(x,S_{m}\right) ^{2}\right) .  \label{5}
\end{eqnarray}
If $\rho (x,S_{\infty })$ exists, then, by the first part of the proof, $\rho \left( x,S_{n}\right) \rightarrow \rho \left( x,S_{\infty }\right) $ as $n\rightarrow \infty $. It follows from this and (\ref{5}) that $\left\Vert x_{m}-x_{n}\right\Vert ^{2}\rightarrow 0$ as $m,n\rightarrow \infty $; whence $\left( x_{n}\right) _{n\geqslant 1}$ is a Cauchy sequence in $H$ and therefore converges to a limit $x_{\infty }\in \overline{S_{\infty }}$. Then
\begin{eqnarray*}
\rho \left( x,S_{\infty }\right) &=&\rho \left( x,\overline{S_{\infty }}\right) \leqslant \left\Vert x-x_{\infty }\right\Vert \\
&=&\lim_{n\rightarrow \infty }\left\Vert x-x_{n}\right\Vert \\
&\leqslant &\lim_{n\rightarrow \infty }\left( \rho \left( x,S_{n}\right)+2^{-n}\right) =\rho \left( x,S_{\infty }\right).
\end{eqnarray*}
Thus $\rho \left( x,S_{\infty }\right) =\left\Vert x-x_{\infty }\right\Vert$.

Conversely, suppose that $x_{\infty }=\lim_{n\rightarrow \infty }x_{n}$ exists. Let $0<\alpha <\beta $ and $\varepsilon =\frac{1}{3}\left( \beta-\alpha \right) $. Pick $N$ such that $2^{-N}<\varepsilon $ and $\left\Vert x_{\infty }-x_{n}\right\Vert <\varepsilon $ for all $n\geqslant N$. Either $\left\Vert x-x_{\infty }\right\Vert >\alpha +2\varepsilon $ or $\left\Vert x-x_{\infty }\right\Vert <\beta $. In the first case, for all $n\geqslant N$,
\begin{eqnarray*}
\rho \left( x,S_{n}\right) &>&\left\Vert x-x_{n}\right\Vert -2^{-n} \\
&\geqslant &\left\Vert x-x_{\infty }\right\Vert -\left\Vert x_{\infty}-x_{n}\right\Vert -\varepsilon \\
&>&\left( \alpha +2\varepsilon \right) -\varepsilon -\varepsilon =\alpha .
\end{eqnarray*}
In the other case, there exists $\nu >N$ such that $\left\Vert x-x_{\nu}\right\Vert <\beta $; we then have
\[\rho \left( x,S_{\nu }\right) \leqslant \left\Vert x-x_{\nu }\right\Vert<\beta.\]
It follows from this and the constructive least-upper-bound principle (\cite{BV}, Theorem 2.1.18) that
\[\inf \left\{ \rho \left( x,S_{n}\right) :n\geqslant 1\right\}\]
exists; whence, by (\ref{6}), $d\equiv \rho \left( x,S_{\infty }\right)$ exists.

Finally, suppose that $x_{\infty }$ exists, and consider any $y\in S_{\infty}$ with $y\neq x_{\infty }$. We have
\begin{eqnarray*}
0 &<&\left\Vert y-x_{\infty }\right\Vert ^{2}=\left\Vert y-x-\left(x_{\infty }-x\right) \right\Vert ^{2} \\
&=&2\left\Vert y-x\right\Vert ^{2}+2\left\Vert x_{\infty }-x\right\Vert^{2}-4\left\Vert \frac{y+x_{\infty }}{2}-x\right\Vert ^{2} \\
&=&2\left( \left\Vert y-x\right\Vert ^{2}-d^{2}\right) +2\left( \left\Vert x_{\infty }-x\right\Vert ^{2}-d^{2}\right) =2\left( \left\Vert y-x\right\Vert ^{2}-d^{2}\right),
\end{eqnarray*}
so $\left\Vert x-y\right\Vert >d$.
\end{proof}

For each positive integer $n$ we write
\[\mathfrak{A}_{n}\equiv n\mathfrak{A}_{1}=\left\{ nA:A\in \mathfrak{A}_{1}\right\}.\]
If $\mathfrak{A}_{1}$ is weak-operator totally bounded and hence strong-operator located, then $\mathfrak{A}_{n}$ has those two properties as well.

Our interest in Proposition \ref{2202b} stems from this:
\begin{cor}\label{2202c}
Let $\mathfrak{A}$ be a linear subspace of $\mathcal{B}(H)$ with $\mathfrak{A}_{1}$ weak-operator totally bounded, and let $x,y\in H$. For each $n$, let $y_{n}\in \mathfrak{A}_{n}$ satisfy $\left\Vert y-y_{n}\right\Vert <\rho \left( x,\mathfrak{A}_{n}x\right) +2^{-n}$. Then
\[\rho \left( y,\mathfrak{A}x\right) =\inf_{n\geqslant 1}\rho (y,\mathfrak{A}_{n}x)=\lim_{n\rightarrow \infty }\rho \left( y,\mathfrak{A}_{n}x\right).\]
Moreover, $\rho \left( y,\mathfrak{A}x\right) $ exists if and only if $\left( y_{n}\right) _{n\geqslant 1}$ converges to a limit $y_{\infty }\in H$; in which case, $\rho \left( y,\mathfrak{A}x\right) =\left\Vert y-y_{\infty}\right\Vert $, and $\left\Vert y-Ax\right\Vert >\left\Vert y-y_{\infty}\right\Vert $ for each $A\in \mathfrak{A}$ such that $Ax\neq y_{\infty}$.
\end{cor}

One case of this corollary arises when the sequence $\left( \rho \left( y,\mathfrak{A}_{n}x\right) \right) _{n\geqslant 1}$ stabilises:
\begin{prop}\label{2202d}
Let $\mathfrak{A}$ be a linear subspace of $\mathcal{B}(H)$ such that $\mathfrak{A}_{1}$ is weak-operator totally bounded. Let $x,y\in H$, and suppose that for some positive integer $N$, $\rho \left( y,\mathfrak{A}_{N}x\right) =\rho \left( y,\mathfrak{A}_{N+1}x\right) $. Then $\rho \left(y,\mathfrak{A}x\right) $ exists and equals $\rho \left( y,\mathfrak{A}_{N}x\right) $.
\end{prop}
\begin{proof}
By Theorem 4.3.1 of \cite{BV}, there exists a unique $z\in \overline{\mathfrak{A}_{N}x}$ such that $\rho \left( y,\mathfrak{A}_{N}x\right)=\left\Vert y-z\right\Vert $. We prove that $y-z$ is orthogonal to $\mathfrak{A}x$. Let $A\in \mathfrak{A}$, and consider $\lambda \in \mathbf{C}$ so small that $\lambda A\in \mathfrak{A}_{1}$. Since,
\[z-\lambda Ax\in \overline{\mathfrak{A}_{N+1}x},\]
we have
\begin{eqnarray*}
\left\langle y-z-\lambda Ax,y-z-\lambda Ax\right\rangle &\geqslant &\rho\left( y,\mathfrak{A}_{N+1}x\right) ^{2} \\
&=&\rho \left( y,\mathfrak{A}_{N}x\right) ^{2}=\left\langle y-z,y-z\right\rangle .
\end{eqnarray*}
This yields
\[\left\vert \lambda \right\vert ^{2}\left\Vert Ax\right\Vert ^{2}+2\func{Re}\left( \lambda \left\langle y-z,Ax\right\rangle \right) \geqslant 0.\]
Suppose that $\func{Re}\left\langle y-z,Ax\right\rangle \neq 0$. Then by taking a sufficiently small real $\lambda$ with 
\[\lambda \func{Re}\left\langle y-z,Ax\right\rangle <0,\] 
we obtain a contradiction. Hence $\func{Re}\left\langle y-z,Ax\right\rangle =0$. Likewise, $\func{Im}\left\langle y-z,Ax\right\rangle =0$. Thus $\left\langle y-z,Ax\right\rangle =0$. Since $A\in \mathfrak{A}$ is arbitrary, we conclude that $y-z$ is orthogonal to $\mathfrak{A}x$ and hence to $\overline{\mathfrak{A}x}$. It is well known that this implies that $z$ is the unique closest point to $y$ in the closed linear subspace $\overline{\mathfrak{A}x}$. Since $\mathfrak{A}x$ is dense in $\overline{\mathfrak{A}x}$, it readily follows that $\rho \left( y,\mathfrak{A}x\right) =\rho \left( y,\overline{\mathfrak{A}x}\right) =\left\Vert y-z\right\Vert $.
\end{proof}

The final result in this section will be used in the proof of our main theorem.

\begin{prop}
\label{2301a2}Let $\mathfrak{A}$ be a linear subspace of $\mathcal{B}(H)$ with weak-operator totally bounded unit ball, and let $x\in H$. Suppose that there exists $r>0$ such that
\[\mathfrak{A}_{1}x\supset B_{\mathfrak{A}x}(0,r)\equiv \mathfrak{A}x\cap B(0,r).\]
Then $\mathfrak{A}x$ is located in $H$; in fact, for each $y\in H$, there exists a positive integer $N$ such that $\rho \left( y,\mathfrak{A}x\right)=\rho \left( y,\mathfrak{A}_{N}x\right) $.
\end{prop}

\begin{proof}
Fixing $y\in H$, compute a positive integer $N>2\left\Vert y\right\Vert /r$. Let $A\in \mathfrak{A}$, and suppose that
\[\left\Vert y-Ax\right\Vert <\rho \left( y,\mathfrak{A}_{N}x\right).\]
We have either $\left\Vert Ax\right\Vert <Nr$ or $\left\Vert Ax\right\Vert>2\left\Vert y\right\Vert $. In the first case, $N^{-1}Ax\in B_{\mathfrak{A}x}(0,r)$, so there exists $B\in \mathfrak{A}_{1}$ with $N^{-1}Ax=Bx$ and therefore $Ax=NBx$. But $NB\in \mathfrak{A}_{N}$, so
\[\left\Vert y-Ax\right\Vert =\left\Vert y-NBx\right\Vert \geqslant \rho\left( y,\mathfrak{A}_{N}x\right),\]
a contradiction. In the case $\left\Vert Ax\right\Vert \geqslant Nr>2\left\Vert y\right\Vert $, we have
\[\left\Vert y-Ax\right\Vert \geqslant \left\Vert Ax\right\Vert -\left\Vert y\right\Vert >\left\Vert y\right\Vert \geqslant \rho \left( y,\mathfrak{A}_{N}x\right),\]
another contradiction. We conclude that $\left\Vert y-Ax\right\Vert \geqslant \rho \left( y,\mathfrak{A}_{N}x\right) $ for each $A\in \mathfrak{A}$. On the other hand, given $\varepsilon >0$, we can find $A\in \mathfrak{A}_{N}$ such that $\left\Vert y-Ax\right\Vert <\rho \left( y,\mathfrak{A}_{N}x\right) +\varepsilon $. It now follows that $\rho \left( y,\mathfrak{A}x\right) $ exists and equals $\rho \left( y,\mathfrak{A}_{N}x\right) $.
\end{proof}

\section{Generalising the open mapping theorem}\label{sec:3}

The key to our main result on the existence of projections of the form $\left[ \mathfrak{A}x\right] $ is a generalisation of the open mapping theorem from functional analysis (\cite{BV}, Theorem 6.6.4). Before giving that generalisation, we note a proposition and a lemma.

\begin{prop}
\label{2802a0}If $C$ is a balanced, convex subset of a normed space $X$, then $V\equiv \dbigcup\limits_{n\geqslant 1}nC$ is a linear subspace of $X$.
\end{prop}
\begin{proof}
Let $x\in V$ and $\alpha \in \mathbf{C}$. Pick a positive integer $n$ and an element $c$ of $C$ such that $x=nc$. If $\alpha \neq 0$, then since $C$ is balanced, $\left\vert \alpha \right\vert ^{-1}\alpha c\in C$, so
\[\alpha x=\alpha nc=\left\vert \alpha \right\vert n\left\vert \alpha\right\vert ^{-1}\alpha c\in \left\vert \alpha \right\vert nC\subset \left(1+\left\vert \alpha \right\vert \right) nC.\]
In the general case, we can apply what we have just proved to show that
\[\left( 1+\alpha \right) x\in \left( 1+\left\vert 1+\alpha \right\vert\right) nC\subset \left( 2+\left\vert \alpha \right\vert \right) nC\text{.}\]
Now, since $C$ is balanced,
\[-x=n\left( -c\right) \in nC\subset (2+\left\vert \alpha \right\vert )nC.\]
Hence, by the convexity of $(2+\left\vert \alpha \right\vert )nC$,
\[\alpha x=2\frac{(1+\alpha )x-x}{2}\in 2(2+\left\vert \alpha \right\vert )nC.\]
Taking $N$ as any integer $>2(2+\left\vert \alpha \right\vert )n$, we now see that $\alpha x\in NC\subset V$. In view of the foregoing and the fact that $\left( nC\right) _{n\geqslant 1}$ is an ascending sequence of sets, if $x^{\prime }$ also belongs to $V$ we can take $N$ large enough to ensure that $\alpha x$ and $x^{\prime }$ both belong to $NC$. Picking $c,c^{\prime}\in C$ such that $\alpha x=Nc$ and $x^{\prime }=Nc^{\prime }$, we obtain
\[\alpha x+x^{\prime }=2N\left( \frac{c+c^{\prime }}{2}\right) \in 2NC,\]
so $\alpha x+x^{\prime }\in V$.
\end{proof}

We call a bounded subset $C$ of a Banach space $X$ \textbf{\emph{superconvex}} if for each sequence $\left( x_{n}\right) _{n\geqslant 1}$ in $C$ and each sequence $\left( \lambda _{n}\right) _{n\geqslant 1}$ of nonnegative numbers such that $\sum_{n=1}^{\infty }\lambda _{n}$ converges to $1$ and the series $\sum_{n=1}^{\infty }\lambda _{n}x_{n}$ converges, we have $\sum_{n=1}^{\infty }\lambda _{n}x_{n}\in C$. In that case, $C$ is clearly convex.

\begin{lem}
\label{2802a2}Let $C$ be a located, bounded, balanced, and superconvex
subset of a Banach space $X$, such that $X=\dbigcup\limits_{n\geqslant 1}nC$%
. Let $y\in X$ and $r>\left\Vert y\right\Vert $. Then there exists $\xi \in
2C$ such that if $y\neq \xi $, then $\rho \left( z,C\right) >0$ for some $z$
with $\left\Vert z\right\Vert <r$.
\end{lem}
\begin{proof}
Either $\rho \left( y,C\right) >0$ and we take $z=y$, or else, as we suppose, $\rho \left( y,C\right) <r/2$. Choosing $x_{1}\in 2C$ such that $\left\Vert y-\frac{1}{2}x_{1}\right\Vert <r/2$ and therefore $\left\Vert 2y-x_{1}\right\Vert <r$, set $\lambda _{1}=0$. Then either $\rho \left(2y-x_{1},C\right) >0$ or $\rho \left( 2y-x_{1},C\right) <r/2$. In the first case, set $\lambda _{k}=1$ and $x_{k}=0$ for all $k\geqslant 2$. In the second case, pick $x_{2}\in 2C$ such that $\left\Vert 2y-x_{1}-\frac{1}{2}x_{2}\right\Vert <r/2$ and therefore $\left\Vert 2^{2}y-2x_{1}-x_{2}\right\Vert <r$, and set $\lambda _{2}=0$. Carrying on in this way, we construct a sequence $\left( x_{n}\right) _{n\geqslant 1}$ in $2C$, and an increasing binary sequence $\left( \lambda _{n}\right)_{n\geqslant 1}$ with the following properties.
\begin{itemize}
\item If $\lambda _{n}=0$, then
\[\rho \left( 2^{n-1}y-\sum_{i=1}^{n}2^{n-i-1}x_{i},C\right) <\frac{r}{2}\]
and
\[\left\Vert 2^{n}y-\sum_{i=1}^{n}2^{n-i}x_{i}\right\Vert <r.\]
\item If $\lambda _{n}=1-\lambda _{n-1}$, then
\[\rho \left( 2^{n-1}y-\sum_{i=1}^{n}2^{n-i-1}x_{i},C\right) >0\]
and $x_{k}=0$ for all $k\geqslant n$.
\end{itemize}
Compute $\alpha >0$ such that $\left\Vert x\right\Vert <\alpha $ for all $x\in 2C$. Then the series $\sum_{i=1}^{\infty }2^{-i}x_{i}$ converges, by comparison with $\left\vert \alpha \right\vert \sum_{i=1}^{\infty }2^{-i}$, to a sum $\xi $ in the Banach space $X$. Since $\sum_{i=1}^{\infty }2^{-i}=1$ and $C$ is superconvex, we see that
\[\sum_{i=1}^{\infty }2^{-i}x_{i}=2\sum_{i=1}^{\infty }2^{-i}\left( \frac{1}{2}x_{i}\right) \in 2C.\]
If $y\neq \xi $, then there exists $N$ such that 
\[\left\Vert y-\sum_{i=1}^{N}2^{-i}x_{i}\right\Vert >2^{-N}r\]
and therefore
\[\left\Vert 2^{N}y-\sum_{i=1}^{N}2^{N-i}x_{i}\right\Vert >r.\]
It follows that we cannot have $\lambda _{N}=0$, so $\lambda _{N}=1$ and therefore there exists $\nu \leqslant N$ such that $\lambda _{\nu}=1-\lambda _{\nu -1}$. Setting
\[z\equiv 2^{\nu -1}y-\sum_{i=1}^{\nu -1}2^{\nu -i-1}x_{i},\]
we see that $\rho (z,C)>0$ and $\left\Vert z\right\Vert <r$, as required.
\end{proof}

We now prove our generalisation of the open mapping theorem.
\begin{thm}
\label{2802a3}Let $X$ be a Banach space,and $C$ a located, bounded, balanced, and superconvex subset of $X$ such that $\rho \left( 0,-C\right)$ exists and $X=\dbigcup\limits_{n\geqslant 1}nC$. Then there exists $r>0$ such that $B\left( 0,r\right) \subset C$.
\end{thm}
\begin{proof}
Consider the identity
\[X=\dbigcup\limits_{n\geqslant 1}\overline{nC}.\]
By Theorem 6.6.1 of \cite{BV} (see also \cite{BHV1}), there exists $N$ such that the interior of $\overline{NC}$ is inhabited. Thus there exist $y_{0}\in NC$ and $R>0$ such that $B\left( y_{0},R\right) \subset \overline{NC}$. Writing $y_{1}=N^{-1}y_{0}$ and $r=\left( 2N\right) ^{-1}R$, we obtain $B\left( y_{1},2r\right) \subset \overline{C}.$It follows from Lemma 6.6.3 of \cite{BV} that $B\left( 0,2r\right) \subset \overline{C}$. Now consider any $y\in B\left( 0,2r\right) $. By Lemma \ref{2802a2}, there exists $\xi \in 2C$ such that if $y\neq \xi $, then there exists $z\in B(0,2r)$ with $\rho\left( z,C\right) >0$. Since $B\left( 0,2r\right) \subset \overline{C}$, this is absurd. Hence $y=\xi \in 2C$. It follows that $B\left( 0,2r\right)\subset 2C$ and hence that $B\left( 0,r\right) \subset C$.
\end{proof}

Note that in Lemma \ref{2802a2} and Theorem \ref{2802a3} we can replace the superconvexity of $C$ by these two properties: $C$ is convex, and for each sequence $\left( x_{n}\right) _{n\geqslant 1}$ in $C$, if $\sum_{n=1}^{\infty }2^{-n}x_{n}$ converges in $H$, then its sum belongs to $C$.

We now derive two corollaries of Theorem \ref{2802a3}.

\begin{cor}[\textbf{The open mapping theorem} (\cite{BV},
Theorem 6.6.4)\label{2802a4}\footnote{This is but one version of the open mapping
theorem; for another, see \cite{BI}.}]
Let $X,Y$ be Banach spaces, and $T$ a sequentially continuous linear mapping of $X$ onto $Y$ such that $T\left( \overline{B(0,1)}\right) $ is located and $\rho \left( 0,-T\left( \overline{B(0,1)}\right) \right) $ exists. Then there exists $r>0$ such that $B\left(0,r\right) \subset T\left( \overline{B\left( 0,1\right) }\right) $.
\end{cor}

\begin{proof}
In view of Theorem \ref{2802a3}, it will suffice to prove that $C\equiv T\left( \overline{B\left( 0,1\right) }\right)$ is superconvex. But if $\left( x_{n}\right) _{n\geqslant 1}$ is a sequence in $\overline{B\left(0,1\right) }$ and $\left( \lambda _{n}\right) _{n\geqslant 1}$ is a sequence of nonnegative numbers such that $\sum_{n=1}^{\infty }\lambda _{n}=1$, then $\left\Vert \lambda _{n}x_{n}\right\Vert \leqslant \lambda _{n}$ for each $n$, so $\sum_{n=1}^{\infty }\lambda _{n}x_{n}$ converges in $X$; moreover,
\[\left\Vert \sum_{n=1}^{\infty }\lambda _{n}x_{n}\right\Vert \leqslant\sum_{n=1}^{\infty }\lambda _{n}=1,\]
so, by the sequential continuity of $T$,
\[T\left( \sum_{n=1}^{\infty }\lambda _{n}x_{n}\right) \in C.\]
Thus $C$ is superconvex.
\end{proof}

Theorem \ref{2802a3} also leads to the \emph{\textbf{proof of Theorem \ref{2202a}:}}

\begin{proof}
Taking $C\equiv \mathfrak{A}_{1}x$, we know that $C$ is located (since $\mathfrak{A}_{1}$ is weak-operator totally bounded and hence, by \cite{BHV,Spitters}, strong-operator located), as well as bounded and balanced. To prove that $C$ is superconvex, consider a sequence$~\left( A_{n}\right)_{n\geqslant 1}$ in $\mathfrak{A}_{1}$, and a sequence $\left( \lambda_{n}\right) _{n\geqslant 1}$ of nonnegative numbers such that $\sum_{n=1}^{\infty }\lambda _{n}$ converges to $1$. For $k\geqslant j$ we have 
\[\left\Vert \sum_{n=j}^{k}\lambda _{n}A_{n}\right\Vert \leqslant\sum_{n=j}^{k}\lambda _{n},\]
so $\sum_{n=1}^{\infty }\lambda _{n}A_{n}$ converges uniformly to an element $A$ of $\mathcal{B}_{1}(H)$. Since $\mathfrak{A}$ is uniformly closed, $A\in \mathfrak{A}_{1}$, so $\sum_{n=1}^{\infty }\lambda _{n}A_{n}x=Ax\in \mathfrak{A}_{1}x$. Thus $C$ is superconvex. We can now apply Theorem \ref{2802a3}, to produce $r>0$ such that $B_{\mathfrak{A}x}\left( 0,r\right)\subset C$. The locatedness of $\mathfrak{A}x$, and the consequent existence of the projection $\left[ \mathfrak{A}x\right] $, now follow from Proposition \ref{2301a2}.\label{ere}
\end{proof}

We now discuss further the requirement, in Theorem \ref{2202a}, that $\rho_{\mathfrak{A}x}\left( 0,-\mathfrak{A}_{1}x\right) $ exist, where $\mathfrak{A}_{1}$ is weak-operator totally bounded. We begin by giving conditions under which that requirement is satisfied.

If $\mathfrak{A}x$ has positive, finite dimension---in which case it is both closed and located in $H$---then $\mathfrak{A}x-\mathfrak{A}_{1}x$ is inhabited, so Proposition (1.5) of \cite{LNM873} can be applied to show that $\mathfrak{A}x-\mathfrak{A}_{1}x$ is located in $\mathfrak{A}x$. In particular, $\rho _{\mathfrak{A}x}\left( 0,-\mathfrak{A}_{1}x\right)$ exists. On the other hand, if $P$ is a projection in $\mathcal{B}(H)$ and
\[\mathfrak{A}\equiv \left\{ PTP:T\in \mathcal{B}(H)\right\},\]
then $\mathfrak{A}$ can be identified with $\mathcal{B}(P(H))$, so $\mathfrak{A}_{1}$ is weak-operator totally bounded. Moreover, if $x\neq 0$, then $\mathfrak{A}x=P(H)$ and so is both closed and located, $\mathfrak{A}_{1}x=\overline{B}(0,\left\Vert Px\right\Vert )\cap P(H)$, and $\rho _{\mathfrak{A}x}(0,-\mathfrak{A}_{1}x)=\left\Vert Px\right\Vert$.

We end with a Brouwerian example showing that we cannot drop the existence of $\rho_{\mathfrak{A}x}\left( 0,-\mathfrak{A}_{1}x\right)$ from the hypotheses of Theorem \ref{2202a}. Consider the case where $H=\mathbf{R}\times \mathbf{R}$, and let $\mathfrak{A}$ be the linear subspace (actually an algebra) of $\mathcal{B}(H)$ comprising all matrices of the form
\[T_{a,b}\equiv \left(\begin{array}{cc}a & 0 \\ 0 & b\end{array}\right)\]
with $a,b\in \mathbf{R}$. It is easy to show that $\mathfrak{A}$ is uniformly closed: if $\left( a_{n}\right) ,\left( b_{n}\right) $ are sequences in $\mathbf{R}$ such that $\left( T_{a_{n},b_{n}}\right)_{n\geqslant 1}$ converges uniformly to an element $T\equiv \left(\begin{array}{cc}a_{\infty } & p \\ q & b_{\infty }\end{array}\right)$, then
\[a_{n}=T_{a_{n},b_{n}}\left(\begin{array}{c}1 \\ 0\end{array}\right) \rightarrow T\left(\begin{array}{c}1 \\ 0\end{array}\right) =a_{\infty },\]
Likewise, $b_{n}\rightarrow b_{\infty }$, $p=0$, and $q=0$. Hence $T=T_{a_{\infty },b_{\infty }}\in \mathfrak{A}$.

Now, if $\left( x,y\right) $ is in the unit ball of $H$, then
\begin{eqnarray*}
\left\Vert T_{a,b}\left(\begin{array}{c}x \\ y\end{array}\right) \right\Vert ^{2} &=&\left\Vert \left(\begin{array}{c}ax \\ by\end{array}\right) \right\Vert ^{2}=a^{2}x^{2}+b^{2}y^{2} \\
&=&a^{2}\left( x^{2}+y^{2}\right) +\left( b^{2}-a^{2}\right) y^{2} \\
&=&a^{2}+\left( b^{2}-a^{2}\right) y^{2}\text{.}
\end{eqnarray*}
We see from this that if $a^{2}\geqslant b^{2}$, then $\left\Vert T_{a,b}\right\Vert ^{2}\leqslant a^{2}$; moreover, $T_{a,b}\left( 1,0\right)=a$, so $\left\Vert T_{a,b}\right\Vert ^{2}=a^{2}$. If $a^{2}<b^{2}$, then a similar argument shows that $\left\Vert T_{a,b}\right\Vert ^{2}=b^{2}$. It now follows that $\left\Vert T_{a,b}\right\Vert $ exists and equals $\max\left\{ \left\vert a\right\vert ,\left\vert b\right\vert \right\} $. Also, since, relative to the uniform topology on $\mathcal{B}(H)$, $\mathfrak{A}_{1}$ is homeomorphic to the totally bounded subset
\[\left\{ \left( a,b\right) :\max \left\{ \left\vert a\right\vert ,\left\vert b\right\vert \right\} \leqslant 1\right\}\]
of $\mathbf{R}^{2}$, it is uniformly, and hence weak-operator, totally bounded.

Consider the vector $\xi \equiv \left( 1,c\right) $, where $c\in \mathbf{R}$. If $c=0$, then $\mathfrak{A}\xi =\mathbf{R}\times \left\{ 0\right\} $, the projection of $H$ on $\mathfrak{A}\xi $ is just the projection on the $x$-axis, and $\rho \left( \left( 0,1\right) ,\mathfrak{A}\xi \right) =1$. If $c\neq 0$, then
\[\mathfrak{A\xi }=\left\{ \left( a,cb\right) :a,b\in \mathbf{R}\right\} =\mathbf{R}\times \mathbf{R},\]
the projection of $H$ on $\mathfrak{A\xi }$ is just the identity projection $I$, and $\rho \left( \left( 0,1\right) ,\mathfrak{A\xi }\right) =0$. Suppose, then, that the projection $P$ of $H$ on $\mathfrak{A\xi }$ exists. Then either $\rho \left( \left( 0,1\right) ,\mathfrak{A\xi }\right) >0$ or $\rho \left( \left( 0,1\right) ,\mathfrak{A\xi }\right) <1$. In the first case, $c=0$; in the second, $c\neq 0$. Thus if $\left[ \mathfrak{A}x\right]$ exists for each $x\in H$, then we can prove that
\[\forall _{x\in \mathbf{R}}\left( x=0\vee x\neq 0\right),\]
a statement constructively equivalent to the essentially nonconstructive omniscience principle \textbf{LPO}:
\begin{quote}
For each binary sequence $\left( a_{n}\right) _{n\geqslant 1}$, either $a_{n}=0$ for all $n$ or else there exists $n$ such that $a_{n}=1$.
\end{quote}
It follows from this and our Theorem \ref{2202a} that if $\rho _{\mathfrak{A}x}\left( 0,-\mathfrak{A}_{1}x\right)$ exists for each $x\in H$, then we can derive \textbf{LPO}.

\section*{Acknowledgement}
  This research was partially done when the author was a visiting fellow at the Isaac Newton Institute for the Mathematical Sciences, in the programme \emph{Semantics \& Syntax: A Legacy of Alan Turing}. The author thanks the referees for helpful comments that improved the presentation of the paper.


\end{document}